\documentclass[11pt,article,reqno]{amsart}
\usepackage{amssymb,url}
\usepackage{graphicx}
\usepackage{amsthm}
\usepackage{amsmath}

\theoremstyle{plain}
\newtheorem{theorem}{Theorem}
\newtheorem{lem}{Lemma}

\newtheorem{defi}{Definition}

\newcommand\arc{\operatorname{arc}}
\newcommand\conv{\operatorname{conv}}

\newcommand{\C}{\mathbb{C}}
\newcommand{\N}{\mathbb{N}}
\newcommand{\Q}{\mathbb{Q}}
\newcommand{\R}{\mathbb{R}}

\begin{document}

\title[The Aleksandrov problem in two dimensions]{A contribution to the Aleksandrov conservative distance problem in two dimensions}

\author{Gy\"orgy P\'al Geh\'er}
\address{Bolyai Institute\\
University of Szeged\\
H-6720 Szeged, Aradi v\'ertan\'uk tere 1, Hungary}
\address{MTA-DE "Lend\"ulet" Functional Analysis Research Group, Institute of Mathematics\\
University of Debrecen\\
H-4010 Debrecen, P.O.~Box 12, Hungary}
\email{gehergy@math.u-szeged.hu; gehergyuri@gmail.com}
\urladdr{\url{http://www.math.u-szeged.hu/~gehergy/}}

\begin{abstract}
Let $E$ be a two-dimensional real normed space. In this paper we show that if the unit circle of $E$ does not contain any line segment such that the distance between its endpoints is greater than 1, then every transformation $\phi\colon E\to E$ which preserves the unit distance is automatically an affine isometry.  In particular, this condition is satisfied when the norm is strictly convex.
\end{abstract}



\maketitle


\section{Introduction}

In 1953 F. S. Beckman and D. A. Quarles characterized isometries of $n$-dimensional Euclidean spaces under a surprisingly mild condition when $n\geq 2$ (see \cite{BeQu} or \cite{Be,Ju} for alternative proofs). Namely, they managed to show that every transformation $\phi\colon\R^n\to\R^n$ which preservers unit Euclidean distance in one direction is an (affine) isometry. They also noted that on $\R$ or on an infinite dimensional, real Hilbert space the same conclusion fails.

Many mathematicians have been trying to generalize this beautiful theorem. The problem of characterizing those finite dimensional real normed spaces $E$ such that every transformation $\phi\colon E\to E$ which preserves the unit distance in one direction is an isometry was raised, in this general form, by A. D. Aleksandrov and hence it is called the Aleksandrov conservative distance problem (see \cite{Al}). In the literature these spaces are also called Beckman-Quarles type spaces. As far as we know, the original version of Aleksandrov problem was solved only for a few concrete normed spaces (see \cite{We} concerning $p$-norms, and \cite{Li} where the norm is not strictly convex), all of them are two-dimensional. Some general results are known for modified versions, for instance in \cite{BeBe} W. Benz and H. Berens investigated the case when the transformation preservers distance 1 and $n$ for some $n\in\N, n>1$. We also mention the paper \cite{RaSe} of T. M. Rassias and P. \v Semrl where they assumed that $\phi$ is onto and it preserves distance 1 in both directions. They showed that in this case $\phi$ is not very far from being an isometry. Several other results are known which are connected to the Aleksandrov problem. The reader can find a number of them in the References.

The original version remained unsolved even for the very special case when $\dim E = 2$ and the norm is strictly convex. Here we present a unified approach which solves the Aleksandrov problem in two dimensions for a much larger class of norms, which we will call URTC-norms. Let us point out that the naive conjecture that every at least two but finite dimensional normed space is a Beckman-Quarles type space is false. However, as far as we know, counterexamples are only known in the simple case when the unit ball of the norm is a linear image of a cube (see \cite{Ra}).


\section{Auxiliary definitions and statement of the main result}

Since we will consider only two-dimensional normed spaces over $\R$, we can investigate $\R^2$ endowed with a norm $\|\cdot\|$. We say that the norm is strictly convex, if its sphere $S$ does not contain any non-degenerated line segment. If three points $a,b,c\in\R^2$ satisfy $d = \|a-b\| = \|b-c\| = \|c-a\|$ for some $d>0$, then these points are said to be in a regular $d$-position. We introduce the following notion.

\begin{defi}
We call $\|\cdot\|$ a URTC-norm (unique regular triangle constructibility) if for every $a,b\in\R^2, \|a-b\| = 1$ the equation system 
\begin{equation}\label{URTCdef_eq}
\left\{\begin{matrix}
\|a-x\| = 1\\
\|b-x\| = 1
\end{matrix}\right.
\end{equation} is satisfied exactly for two points $x \in \R^2$.
\end{defi}

Since the function $f(x) := \|b-x\|$ is continuous on $a+S$, $f(b) =0$ and $f(2a-b)=2$, the existence of such an $x$ which fulfilles \eqref{URTCdef_eq} is trivial. Obviously, if $x$ satisfies \eqref{URTCdef_eq}, then $a+b-x \neq x$ fulfilles it as well. 

By translation, we may have assumed that $a = 0$, and by multiplying with a non-zero scalar, we may have replaced 1 by any $d > 0$ in Definition 1. We note that for the $\ell^\infty$ norm, one can find two points $a,b\in\R^2$ with $\|a-b\| = 1$ such that \eqref{URTCdef_eq} holds for infinitely many points $x\in\R^2$. We will provide a useful characterization of URTC-norms in Lemma \ref{URTC_lem}.

Our main theorem, which reads as follows and will be proven in Section \ref{proof_sec}, provides an affirmative answer for the Aleksandrov conservative distance problem for URTC-noms.

\begin{theorem}\label{main_thm}
Let $\|\cdot\|$ be a URTC-norm on $\R^2$, and let us consider an arbitrary transformation $\phi\colon\R^2\to\R^2$ such that
\[
x,y\in \R^2,\; \|x-y\| = 1 \quad \Longrightarrow \quad \|\phi(x)-\phi(y)\| = 1.
\]
Then $\phi$ is an affine isometry.
\end{theorem}

We will need several lemmas before proving Theorem \ref{main_thm}. We note that Theorem \ref{main_thm} can be considered as a Mazur-Ulam type result in two dimensions (see \cite{MaUl,FlJa}). Let us point out that quite the same proof works for the case if we consider two different URTC-norms on the initial and final spaces. However, dealing with the above version makes notations much simpler. Furthermore, by affinity, the modified version of our main theorem says in many cases (in fact when the unit circles of these norms are not linear images of each other) that no transformation $\phi$ exists which preserves the unit distance.


\section{Proof of the main result}\label{proof_sec}

We begin with a characterization of URTC-norms. The symbols $[a,b]$ and $\ell(a,b)$ will denote the line segment $\{a+t(b-a)\colon 0\leq t \leq 1\}$ and the line $\{a+t(b-a)\colon t\in\R\}$, respectively.

\begin{lem}\label{URTC_lem}
The following conditions are equivalent for any norm $\|\cdot\|$ on $\R^2$:
\begin{itemize}
\item[(i)] $\|\cdot\|$ is not a URTC-norm,
\item[(ii)] two points $c,d\in S, \|c-d\|>1$ exist such that $[c,d]\subseteq S$.
\end{itemize}
In particular, every strictly convex norm is a URTC-norm.
\end{lem}

\begin{proof}
(ii)$\Longrightarrow$(i): Set $a=0$ and $b = \frac{1}{\|d-c\|}(d-c) \in S$. Then every $x\in[c+b,d]$ satisfies \eqref{URTCdef_eq}. Therefore the norm cannot be URTC.

(i)$\Longrightarrow$(ii): Since the norm does not have the URTC property, there exists a point $b\in S$ such that at least three different solutions can be given which satisfy \eqref{URTCdef_eq} with $a = 0$. Clearly, none of them can lie on $\ell(0,b)$. Therefore at least two of them, $x$ and $y$, lies on the same open side of $\ell(0,b)$. First, we show that  $x\in\ell(y,y-b)$. Assuming the contrary, we can suppose, without loss of generality, that $x$ lies  between $\ell(0,b)$ and $\ell(y,y-b)$. Clearly $\conv(b,y,y-b,-b)\cap S \subseteq [b,y]\cup[y,y-b]\cup[y-b,-b]$, but since $\|x-b\| = \|y-b\|$ holds, we get $x\notin[b,y]$. On the one hand, if $x$ lies on the opposite closed side of $\ell(-b,y-b)$ than $b$, then $x$ has to be in the interior of $\conv(0,y,x-b)$. Since $y,x,x-b\in S$, this is impossible. On the other hand, if $x$ lies on the opposite side of $\ell(b,y)$ than $0$, then $x-b$ has to be in the interior of $\conv(0,x,y-b)$ which is again a contradiction. Therefore, indeed $x\in\ell(y,y-b)$ is satisfied.

Now, we may suppose that $x-y = \|x-y\|b$. Since $y-b,y,x$ are distinct collinear points of the unit sphere, we easily obtain that $[y-b,x]\subseteq S$ and $\|(y-b)-x\|>1$, which completes the proof of this part.
\end{proof}

The shorter closed and open arcs of $S$ between two non-antipodal points $b_0,b_1$ will be denoted by $\arc(b_0,b_1)$ and $\arc^{\circ}(b_0,b_1)$, respectively. We provide some basic properties of URTC-norms in the following lemma.

\begin{lem}\label{basic_lem}
\begin{itemize}
\item[(i)] Let $\|\cdot\|$ be a URTC-norm. If $a,b,c,d\in S$ such that $\|a-d\| = \|b-c\| = 1$ and $(a-d,a,b,c,d,d-a)$ is positively oriented, then we have $a=b$ and $c=d$.
\item[(ii)] Let $\|\cdot\|$ be an arbitrary norm. Suppose that we have $a,b\in\R^2,\; 0 < \gamma := \|a-b\|$ and $0\leq \alpha,\beta$. If $|\beta-\gamma|\leq\alpha\leq\beta+\gamma$ is satisfied, then there exists a point $c\in\R^2$ which fulfilles $\|a-c\| = \beta$ and $\|b-c\| = \alpha$.
\end{itemize}
\end{lem}

\begin{proof}
(i): Clearly we have $\arc^\circ(a,d) \subseteq \conv(a,d,a+d)$, and by the URTC property $\arc^\circ(a,d)\cap([a,a+d]\cup[d,a+d]) = \emptyset$ is valid. Elementary observations show that if the Euclidean distance between $b$ and $\ell(a,d)$ is less than or equal to the distance of $c$ and $\ell(a,d)$, then $c-b\in\conv(0,d-a,d)$ holds. This implies $c-b\in[d-a,d]$. Thus we obtain $b = a$ and $d\in [d-a,c]$, and hence $[d-a,c]\subseteq S$ is satisfied. Since our norm is URTC, we obtain $c=d$. The other case can be shown similarly.

(ii): We consider the continuous function $F\colon \beta\cdot S +a\to\R,\; F(z) = \|b-z\|$. Since we have $F\left(a+\frac{\beta}{\gamma}(b-a)\right) = |\beta-\gamma| \leq \alpha$ and $F\left(a-\frac{\beta}{\gamma}(b-a)\right) = \beta+\gamma \geq \alpha$, we conclude the existence of a point $c\in\R^2$ such that $\|a-c\| = \beta$ and $\|b-c\| = \alpha$ holds.
\end{proof}

We define the functions $f,g\colon S\to S, z\mapsto f(z)$ such that $\|z-f(z)\| = 1$, $\|z-g(z)\| = 1$, and $(0,z,f(z)), (0,g(z),z)$ are positively oriented. By the URTC property, $f$ and $g$ are well-defined, moreover, we clearly have $g^{-1} = f$, and hence $f$ and $g$ are bijective. We proceed with showing that $f$ is continuous.

\begin{lem}
The function $f$ is continuous.
\end{lem}

\begin{proof}
We assume indirectly that $f$ is not continuous, and without loss of generality we may suppose that $b_1 \in S$ is a point of discontinuity. We set $b_0 = g(b_1)$ and $b_2 = f(b_1)$. 

A quite straightforward application of Lemma \ref{basic_lem} gives the following monotonicity property of $f$: if $z\in\arc^\circ(z_0,f(z_0))$, then we have $f(z)\in\arc^\circ(f(z_0),f(f(z_0)))$. By this monotonicity property we obtain that there are two points $b_{2-\varepsilon} \in \arc^\circ(b_1,b_2)$ and $b_{2+\varepsilon} \in \arc^\circ(b_2,-b_0)$ such that either $f(S)\cap\arc^{\circ}(b_{2-\varepsilon},b_2) = \emptyset$, or $f(S)\cap\arc^{\circ}(b_2,b_{2+\varepsilon}) = \emptyset$. Both of them contradicts to the bijectivity of $f$.
\end{proof}

Let $d>0$. We call the 7-tupple $(a,b_1,b_2,b_3,c_1,c_2,c_3)\in(\R^2)^7$ a $d$-probe if 
\begin{equation}\label{dprobe_eq}
\begin{gathered}
d = \|a-b_1\| = \|a-b_2\| = \|b_1-b_2\| = \|b_1-b_3\| = \|b_2-b_3\|\\
= \|a-c_1\| = \|a-c_2\| = \|c_1-c_2\| = \|c_1-c_3\| = \|c_2-c_3\| = \|b_3-c_3\|
\end{gathered}
\end{equation}
holds (\cite{Ra2,Ju}, see also Figure 1). By the following lemma, any three points which are in a regular $d$-position can be extended to a $d$-probe.

\begin{figure}[h]
\centering
\includegraphics[bb=0 0 495 423,scale=0.35]{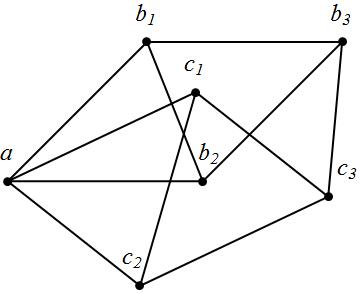}
\caption{A $d$-probe. An edge represents that the distance between the endpoints is exactly $d$.}
\end{figure}

\begin{lem}\label{probe_lem}
Let $\|\cdot\|$ be a URTC-norm, and let $b_1,b_2\in d\cdot S$ such that $\|b_1-b_2\| = d$. Then the 3-tupple $(0,b_1,b_2)$ can be extended to a $d$-probe $(0,b_1,b_2,b_3,c_1,c_2,c_3)$, where we necessarily have $b_3 = b_1+b_2$. Moreover, $0\neq b_3$ holds.
\end{lem}

\begin{proof}
We can assume without loss of generality that $(0,b_1,b_2)$ is positively oriented. Let us define
\[
h\colon d\cdot S\to\R, \quad z \mapsto \left\|\frac{1}{2}(z+f(z)) - \frac{1}{2}(b_1+b_2)\right\|,
\]
which is trivially continuous and $h(b_1) = 0$. Moreover, by the triangle inequality we obtain
\[
h(-b_1) = \left\|\frac{1}{2}(-b_1-b_2) - \frac{1}{2}(b_1+b_2)\right\| = \|b_1 + b_2\| 
\]
\[
= \|2b_2 - (b_2-b_1)\| \geq \|2b_2\| - \|(b_2-b_1)\| = d,
\]
which immediately implies the existence of a $c_1 \in d\cdot S$ such that $h(c_1) = \frac{d}{2}$. We define $b_3 = b_1 + b_2$, $c_2 = f(c_1)$ and $c_3 = c_1 + c_2$. The ordered 7-tupple $(0,b_1,b_2,b_3,c_1,c_2,c_3)$ is trivially a $d$-probe.

In the above construction we chose $b_3$ to be $b_1+b_2$. Now, we show that this is the only choice. Since the norm is URTC, the first line of \eqref{dprobe_eq} (with $a = 0$) implies $b_3\in\{0,b_1+b_2\}$. Suppose that $0 = b_3$ happens. Since $\|b_3-c_3\| = d$, we obtain $c_3 = c_1+c_2$. But from $c_1-c_2,c_1,c_1+c_2\in d\cdot S$ we conclude $[c_1-c_2,c_1+c_2]\subseteq d\cdot S$, which clearly contradicts the URTC property. Therefore we indeed have $b_3 = b_1+b_2$.
\end{proof}

Now, we are in the position to present the proof of the main result of this paper.

\begin{proof}[Proof of Theorem \ref{main_thm}]
Suppose that $\phi$ preserves distance $d > 0$. Let $b_0,b_1,b_2\in \R^2$ be arbitrary three points which are in a regular $d$-position. Let us consider a $d$-probe $(b_0,b_1,b_2,b_3,c_1,c_2,c_3)\in(\R^2)^7$ which exists by Lemma \ref{probe_lem} and where $b_3 = b_1+b_2-b_0$. Clearly, the 7-tupple $(\phi(b_0),\phi(b_1),\phi(b_2),\phi(b_3),$ $\phi(c_1),\phi(c_2),\phi(c_3))\in(\R^2)^7$ is a $d$-probe as well, and therefore $\phi(b_3) = \phi(b_1+b_2-b_0) = \phi(b_1) + \phi(b_2) - \phi(b_0)$ is satisfied. Set $b_4 = 2b_2-b_0$. Considering $b_1,b_2,b_3$ instead of $b_0,b_1,b_2$, by the previous observations we conclude $\phi(b_4) = \phi(2b_2-b_0) = 2\phi(b_2) - \phi(b_0)$. This immediately implies that distance $2d$ is also preserved, moreover, when $b_0,b_2,b_4$ are collinear such that $d = \|b_0-b_2\| = \|b_2-b_4\|$ and $2d = \|b_0-b_4\|$, then the same hold for their images.

\begin{figure}[h]
\centering
\includegraphics[bb=0 0 360 124,scale=0.5]{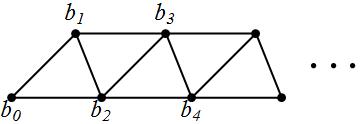}
\caption{}
\end{figure}

Iterating the above method we can easily prove also the following statement: distance $nd$ is also preserved for every $n\in\N$, furthermore, when $a,b,c\in\R^2$ are collinear such that $d = \|a-b\|, (n-1)d = \|b-c\|$ and $nd = \|a-c\|$, then the same is valid for their images. In particular, $\phi$ preserves distance $n$ for all $n\in\N$ (see Figure 2).

\begin{figure}[h]
\centering
\includegraphics[bb=0 0 360 159,scale=0.5]{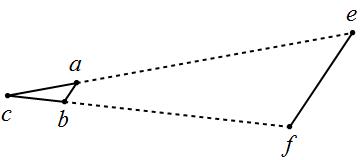}
\caption{}
\end{figure}

Next, we show that assuming distance $d$ is preserved implies that distance $\frac{d}{n}$ is preserved as well for every $n\in\N$. Let $a,b\in \R^2, \|a-b\| = \frac{d}{n}$. By Lemma \ref{basic_lem}, there exists a point $c\in\R^2$ such that $\|a-c\| = \|b-c\| = d$. Set $e = c + n(a-c)$ and $f = c + n(b-c)$. Obviously, we have $\|a-e\| = \|b-f\| = (n-1)d$ and $\|e-f\| = d$. Therefore we have $d = \|\phi(a)-\phi(c)\| = \|\phi(b)-\phi(c)\| = \|\phi(e)-\phi(f)\|$, $(n-1)d = \|\phi(a)-\phi(e)\| = \|\phi(b)-\phi(f)\|$, moreover, $\phi(c),\phi(a),\phi(e)$ are collinear and $\phi(c),\phi(b),\phi(f)$ are collinear. This implies $\phi(a)-\phi(b) = \frac{1}{n}(\phi(e)-\phi(f))$, and thus $\|\phi(a)-\phi(b)\| = \frac{1}{n}\|\phi(e)-\phi(f)\| = \frac{d}{n}$. (See Figure 3).

By the above observations, we immediately obtain that $\phi$ preserves all rational distances. Let $a,b\in\R^2$ be two arbitrary different points. For every $0 < \varepsilon < \frac{\|a-b\|}{3}$ we can find $p,q\in\Q$ such that $0 < q < \varepsilon$ and $p-q < \|a-b\| < p+q$. By Lemma \ref{basic_lem} we can find such a point $c\in\R^2$ which satisfies $\|a-c\| = p, \|b-c\| = q$. Since rational distances are preserved by $\phi$, we get $\|\phi(a)-\phi(c)\| = p, \|\phi(b)-\phi(c)\| = q$, and by the triangle inequality
\[ p-\varepsilon < p-q \leq \|\phi(a)-\phi(b)\| \leq p+q <p+\varepsilon. \]
Since this holds for every $0 < \varepsilon < \frac{\|a-b\|}{3}$, we conclude $\|\phi(a)-\phi(b)\| = \|a-b\|$, which means that $\phi$ is indeed an isometry.

Since isometries are continuous, affinity of $\phi$ follows from the preservations of midpoints which was pointed out before. This completes the proof.
\end{proof}

Several norms on $\R^2$ do not have the URTC property. It is not clear what the answer is for the Aleksandrov conservative distance problem for these norms. Those techniques which were presented here do not work for these class of norms. We left this question as a challenging open problem.


\section*{Acknowledgements.}
The author was supported by the "Lend\" ulet" Program (LP2012-46/2012) of the Hungarian Academy of Sciences.

\end{document}